\documentclass{amsart}
\usepackage{amssymb,latexsym}
\usepackage{times}
\usepackage[usenames]{color}

\newcommand{\divv}{\text{div}}
\newcommand{\ord}{\text{ord}}

\theoremstyle{plain}

\newtheorem{lemma}{Lemma}
\newtheorem{proposition}{Proposition}
\theoremstyle{definition}

\newtheorem*{example}{Example}
\newtheorem*{remark}{Remark}
\theoremstyle{remark}

\numberwithin{equation}{section}

\newcommand{\nocontentsline}[3]{}
\newcommand{\tocless}[2]{\bgroup\let\addcontentsline=\nocontentsline#1{#2}\egroup}

\begin{document}
\title[Beyond two criteria for supersingularity]{Beyond two criteria for supersingularity: \\ coefficients of division polynomials}
\author{Christophe~Debry}
\address{Mathematics Department, KU Leuven, Celestijnenlaan 200B, 3001 Leuven, Belgium
\emph{and} Mathematisch Instituut, Leiden University.}
\email{christophe.debry@wis.kuleuven.be}
\thanks{The author is supported by a PhD fellowship of the Research Foundation -- Flanders (FWO)}
\keywords{Elliptic curves, division polynomials, supersingularity, coefficients}
\begin{abstract}
Let $E: y^2 = x^3+Ax+B$ be an elliptic curve defined over a finite field of characteristic $p\geq 3$. In this paper we prove that the coefficient at $x^{\frac{1}{2}p(p-1)}$ in the $p$--th division polynomial $\psi_p(x)$ of $E$ equals the coefficient at $x^{p-1}$ in $(x^3+Ax+B)^{\frac{1}{2}(p-1)}$. The first coefficient is zero if and only if the division polynomial has no roots, which is equivalent to $E$ being supersingular. Deuring (1941) proved that this supersingularity is also equivalent to the vanishing of the second coefficient. So the zero loci of the coefficients (as functions of $A$ and $B$) are equal; the main result in this paper is clearly stronger than this last statement.
\end{abstract}
\maketitle


\section*{Introduction}

Let $\mathbb{F}_{p^k}$ be a finite field of characteristic $p\geq 3$ and let $E/\mathbb{F}_{p^k}$ be an elliptic curve given by a short Weierstrass equation $E:y^2 = x^3+Ax+B$. Associated to $E$, one defines division polynomials $\psi_m$ (for every positive integer $m$), whose properties we shall review in Section 1. These polynomials can be used to check whether $E$ is supersingular or not:

\medskip
\medskip

\begin{center}\begin{minipage}{10.5cm} \textsc{Division polynomial criterion} \\ $E$ is supersingular if and only if the coefficient at $x^{\frac{1}{2}p(p-1)}$ in $\psi_p$ is zero. \end{minipage}\end{center}

\medskip
\medskip

\noindent For example, let $E: y^2 = x^3+Ax+B$ be an elliptic curve over $\mathbb{F}_{5^k}$. Then $\psi_5$ is equal to $2Ax^{10} + 4A^2Bx^5 + \left(4B^4-2A^3B^2+A^6\right)$. So $E$ is supersingular if and only if $A = 0$. There is also a classical criterion, very similar (in wording) to the one above.


\medskip
\medskip

\begin{center}\begin{minipage}{10.5cm} \textsc{Deuring criterion} \\ Let $E:y^2 = f(x)$ be an elliptic curve over $\mathbb{F}_{p^k}$, where $f(x)\in \mathbb{F}_{p^k}[x]$ is a cubic polynomial with distinct roots in $\overline{\mathbb{F}}_{p^k}$. Then $E$ is supersingular if and only if the coefficient of $x^{p-1}$ in $f(x)^{(p-1)/2}$ is zero.\end{minipage}\end{center}

\medskip
\medskip

\noindent For a proof of this criterion, one can consult Silverman \cite[V.4.1]{Silv}. We reconsider the above example: an elliptic curve $E:y^2 = x^3+Ax+B$ over $\mathbb{F}_{5^k}$ is supersingular if and only if the coefficient at $x^4$ in $(x^3+Ax+B)^2$ is zero, i.e., if and only if $2A = 0$. This is indeed the same criterion as the one we got using division polynomials. The striking similarity between the criteria actually has a deeper reason: not only do these coefficients at different monomials in different polynomials have the same zeros, they actually are equal, as we prove in Section 2. More precisely, we prove the following theorem:

\newpage

\noindent \textbf{Theorem.} \emph{Consider the elliptic curve $E:y^2 = x^3+Ax+B$ over $\mathbb{Q}(A,B)$ (where $A$ and $B$ are transcendentals). Let $p\geq 3$ be prime and let $\ell_p(A,B)$ be the coefficient at $x^{\frac{1}{2}p(p-1)}$ in the $p$--th division polynomial of $E$. Let $c_p(A,B)$ be the coefficient at $x^{p-1}$ in $(x^3+Ax+B)^{\frac{1}{2}(p-1)}$. Then $\ell_p(A,B) \equiv c_p(A,B)\pmod{p}$.}

\section{Division polynomials}


Let $\mathbb{F}_{p^k}$ be a finite field of characteristic $p\geq 3$, with $p^k$ elements. Let $E/\mathbb{F}_{p^k}$ be an elliptic curve with Weierstrass model $$E:y^2+a_1xy+a_3y = x^3+a_2x^2+a_4x+a_6.$$We denote the neutral element of the group law on $E$ by $\mathcal{O}$, and denote the multiplication--by--$m$ isogeny by $[m]$. The division polynomials $(\psi_m)_{m\geq 1}$ associated to $E$ are defined by recursion: $$\begin{aligned} \psi_1 & = 1 \quad \quad \psi_2 = 2y+a_1x+a_3 \quad\quad \psi_3 = 3x^4+b_2x^3+3b_4x^2+3b_6x+b_8, \\ \psi_4 & = \psi_2\cdot\left(2x^6+b_2x^5+5b_4x^4+10b_6x^3+10b_8x^2+(b_2b_8-b_4b_6)x+(b_4b_8-b_6^2)\right),\end{aligned}$$and
$$\begin{aligned} \psi_{2m+1} & = \psi_{m+2}\psi_m^3-\psi_{m-1}\psi_{m+1}^3, \label{eq:odd} \\ \psi_2\psi_{2m} & = \psi_{m-1}^2\psi_m\psi_{m+2}-\psi_{m-2}\psi_m\psi_{m+1}^2. \label{eq:even}\end{aligned}$$Recall that the $b$--quantities used in $\psi_3$ and $\psi_4$ are polynomials in the $a$--quantities: $b_2 = a_1^2+4a_2$, $b_4 = 2a_4+a_1a_3$, $b_6 = a_3^2+4a_6$ and $b_8 = a_1^2a_6+4a_2a_6-a_1a_3a_4+a_2a_3^2-a_4^2$. Every $\psi_m\in\mathbb{F}_{p^k}[x,y]$ can be written as a linear polynomial in $y$ over $\mathbb{F}_{p^k}[x]$ using the Weierstrass equation. As such, one can prove that if $m$ is odd, then $\psi_m\in \mathbb{F}_{p^k}[x]$, and as a polynomial in $x$, $\psi_m$ has degree at most $\frac{1}{2}(m^2-1)$ and the coefficient at $x^{\frac{1}{2}(m^2-1)}$ is equal to $m$. In particular, since we assume $p$ to be an odd prime, the polynomial $\psi_p\in\mathbb{F}_{p^k}[x]$ has degree strictly smaller than $\frac{1}{2}(p^2-1)$. The proofs of these claims can be found in various places, e.g., \cite[3.6]{Enge}. We will also need the following standard facts:
\begin{itemize}
\item The roots of $\psi_m$ are precisely the nontrivial $p$--torsion points on $E$, i.e., the points $P\in E(\overline{\mathbb{F}}_{p^k})\setminus\{\mathcal{O}\}$ satisfying $[p]P = \mathcal{O}$.
\item The polynomials $\psi_m^2$ and $\phi_m = x\psi_m^2 - \psi_{m-1}\psi_{m+1}$ can be considered as elements of $\mathbb{F}_{p^k}[x]$ using the Weierstrass equation, and as such are relatively prime.
\item Denoting the Weierstrass $x$--coordinate function on $E$ by $x$, the functions $x\circ[m]$ and $\phi_m/\psi_m^2$ on $E$ are equal.
\end{itemize}

\noindent We can deduce the following crucial result about the $p$--th division polynomial in characteristic $p\geq 3$.

\begin{proposition}
Let $E/\mathbb{F}_{p^k}$ be an ordinary elliptic curve ($p\geq 3$ prime). Then $\psi_p$ has degree $\frac{1}{2}p(p-1)$ and lies in $\mathbb{F}_{p^k}[x^p]$. \label{prop:xp}
\end{proposition}
\begin{proof}
Note that $[p]$ is not separable and hence factors through the $p$--th power Frobenius $$\Phi:E\to E^{(p)}:[X:Y:Z]\mapsto [X^p:Y^p:Z^p],$$ where $E^{(p)}$ is the elliptic curve defined by the Weierstrass equation with coefficients $a_i^p$. (Cf. \cite[II.2.12]{Silv}) It follows that $x\circ [p]$ is a rational function of $x^p$ and $y^p$. Since finite fields are perfect, this implies that $x\circ [p]$ is the $p$--th power of a rational function in $x$ and $y$. So the coefficients of the divisor of $x\circ [p]$ are all divisible by $p$. Since $x\circ [p] = \phi_p/\psi_p^2$ where $\phi_p$ and $\psi_p^2$ are coprime, we find that the coefficients in $2\divv(\psi_p)$ are $p$--divisible. The zero set $\mathcal{Z}$ of $\psi_p$ is equal to $(\text{ker}[p])(\overline{\mathbb{F}}_{p^k})\setminus\{\mathcal{O}\}$, and $\psi_p$ has only a pole at $\mathcal{O}$, so $$\divv(\psi_p) = \sum_{P\in \mathcal{Z}} n_P\langle P\rangle - n\langle \mathcal{O}\rangle,$$where $n = \sum_{P\in \mathcal{Z}} n_P$ and each $n_P \geq 1$. By the $p$--divisibility of the coefficients, we get that $p$ divides each $2n_P$ and therefore divides each $n_P$ ($p$ is odd). It follows that $n_P\geq p$ and $n \geq p\cdot\sharp\mathcal{Z} = p(p-1)$ because $E$ is ordinary. The polynomial $\psi_p\in\mathbb{F}_{p^k}[x]$ has degree $\leq \frac{1}{2}(p^2-1)$ and hence has order at least $1-p^2$ in $\mathcal{O}$. In other words, $-n\geq 1-p^2$, which together with $p\mid n$ implies that $n\leq p(p-1)$. We find that $n = p(p-1)$ and hence $$\divv(\psi_p) = \sum_{P\in \mathcal{Z}} p\langle P\rangle - p(p-1)\langle \mathcal{O}\rangle = p\left(\sum_{P\in \mathcal{Z}} \langle P\rangle - (p-1)\langle \mathcal{O}\rangle\right).$$The first implication is that the degree of $\psi_p\in\mathbb{F}_{p^k}[x]$ is equal to $-\frac{1}{2}\ord_{\mathcal{O}}(\psi_p) = \frac{1}{2}p(p-1)$. One also easily verifies that the sum of the points in $\mathcal{Z}$ is equal to $\mathcal{O}$, so the divisor $\frac{1}{p}\divv(\psi_p)$ is principal. Therefore, $\psi_p$ is the $p^{\text{th}}$ power of a polynomial in $\mathbb{F}_{p^k}[x]$, which (working in characteristic $p$) implies that $\psi_p\in\mathbb{F}_{p^k}[x^p]$.
\end{proof}

\begin{remark} An alternative to prove this proposition is to use the main theorem from \cite{Cass}. Cheon and Hahn \cite{ChHa} prove the proposition for ordinary elliptic curves over the prime field $\mathbb{F}_p$.\end{remark}

\begin{example} Let $E: y^2 = x^3+Ax+B$ be an elliptic curve over $\mathbb{F}_{5^k}$. Then $\psi_5$ is equal to $2Ax^{10} + 4A^2Bx^5 + \left(4B^4-2A^3B^2+A^6\right)$. Note that $\psi_5$ is indeed a function of $x^5$. It also follows from the proposition that if $E$ is ordinary, then $\psi_5$ must have degree $5\cdot 4/2 = 10$, so $A \neq 0$ if $E$ is ordinary. \end{example}

We can now derive the division polynomial criterion for supersingularity. Let $E/\mathbb{F}_{p^k}$ be an elliptic curve. Since the zeros of $\psi_p$ are precisely the nontrivial $p$--torsion points, $E$ is supersingular if and only if $\psi_p$ has no zeros, i.e., $\psi_p$ is a constant polynomial. This is equivalent to all nonconstant coefficients of $\psi_p$ being zero and this means we have $O(p^2)$ equations to be satisfied. (Indeed, $p$ is odd, so $\psi_p$ can be written as a polynomial in $x$ of degree at most $\frac{1}{2}(p^2-1)$.) But we know that if $E$ is ordinary, then $\psi_p$ has degree $\frac{1}{2}p(p-1)$. This implies that $E$ is supersingular if and only if the coefficient at $x^{\frac{1}{2}p(p-1)}$ in $\psi_p$ is zero, which is the division polynomial criterion mentioned in the introduction.

\begin{example} Reconsider the previous example. Then $\psi_5$ is constant if and only if $2A = 4A^2B = 0$, which indeed is equivalent to $2A = 0$. In other words: $E$ is supersingular if and only if $A = 0$. Note that we went from $12 = (5^2-1)/2$ equations (in characteristic zero, or when we want to work over $\mathbb{Z}[A,B,x,y]$, we need all the nonconstant coefficients to be zero) to $(5-1)/2 = 2$ equations (because $\psi_5$ turned out to be a function of $x^5$), to just one equation. 
\end{example}

\section{Proof of the Theorem}

Let us first fix some notation. Let $A$ and $B$ be indeterminates and consider the sequence of polynomials in $\mathbb{Z}[x,y,A,B]$ defined by $$\begin{aligned} \psi_0 & = 0 \\ \psi_1 & = 1 \\ \psi_2 & = 2y \\ \psi_3 & = 3x^4+6Ax^2+12Bx-A^2, \\ \psi_4 & = 2y\left(2x^6+10Ax^4+40Bx^3-10A^2x^2-8ABx-2(A^3+8B^2)\right),\end{aligned}$$the relation $y^2 = x^3+Ax+B$, and the recursion formulas $$\begin{aligned} \psi_{2m+1} & = \psi_{m+2}\psi_m^3-\psi_{m-1}\psi_{m+1}^3, \\ 2y\psi_{2m} & = \psi_{m-1}^2\psi_m\psi_{m+2}-\psi_{m-2}\psi_m\psi_{m+1}^2.\end{aligned}$$
One can easily prove that $\psi_m\in \mathbb{Z}[x,A,B]$ if $m$ is odd, so we write $\psi_p(x,A,B)$ to denote the $p$--th polynomial in this sequence. Now define $\ell_p(A,B)$ to be the coefficient at $x^{\frac{1}{2}p(p-1)}$ in $\psi_p(x,A,B)\in\mathbb{Z}[x,A,B]$. Define $c_p(A,B)$ as the coefficient at $x^{p-1}$ in $(x^3+Ax+B)^{\frac{1}{2}(p-1)}$. For example, $\ell_5(A,B) = 62A$ (we are not yet reducing mod 5) and $c_p(A,B) = 2A$.

\medskip
\medskip

\noindent \textbf{Theorem.} \emph{
Let $p\geq 3$ be a prime number. Then $c_p(A,B) \equiv \ell_p(A,B)\pmod{p}$. \label{prop:equal}
}

\medskip
\medskip

The remainder of this section consists of the proof of the theorem. To simplify notations, write $p = 2q+1$ with $q\in\mathbb{Z}$. One can easily check the theorem for $p = 3$: both coefficients are zero. So suppose $p\geq 5$ from now on.

\subsection{Step 1: $c_p(A,B)$ as a sum}

First, we compute $c_p(A,B)$ by using Newton's trinomial identity: $$(x^3+Ax+B)^q = \sum_{(i,j,k)\in S} \binom{q}{i,j,k} x^{3i+j}A^jB^k,$$where $S = \left\{(i,j,k)\in\mathbb{Z}^3\mid i,j,k\geq 0, i+j+k = q\right\}$ and $$\binom{q}{i,j,k} = \frac{q!}{i!j!k!}.$$Hence, $$c_p(A,B) = \sum_{(i,j,k)\in S_0} \binom{q}{i,j,k}A^jB^k,$$where $S_0 = \{(i,j,k)\in S\mid 3i+j = p-1 = 2q\}$. Let us determine $S_0$ more explicitly. The triple $(i,j,k)$ is in $S_0$ if and only if $i = \frac{1}{3}(2q-j)$, $k = q-i-j = \frac{1}{3}(q - 2j)$, and $i,j,k$ are non--negative integers. So $$S_0 = \left\{\left(\frac{1}{3}(2q-j),j,\frac{1}{3}(q-2j)\right)\mid j\equiv -q\pmod{3}, j\in\mathbb{Z}\cap \left[0,\frac{q}{2}\right]\right\}.$$ We find that $$c_p(A,B) = \sum_{j\in J} \binom{q}{\frac{1}{3}(2q-j),j,\frac{1}{3}(q-2j)}A^jB^{\frac{1}{3}(q-2j)},$$where $J = \left\{j\in\mathbb{Z}\mid j\equiv -q\pmod{3}, 0 \leq j \leq \frac{1}{2}q\right\}$.

\subsection{Step 2: $\ell_p(A,B)$ as a sum}

Write $\psi_p(x,A,B) = \sum_t \beta_t(A,B)x^t$, with $\beta_t(A,B)\in\mathbb{Z}[A,B]$. Note that if we give $x$ degree 1, $A$ degree 2 and $B$ degree 3, then $y^2 = x^3+Ax+B$ is homogeneous of degree 3, so giving $y$ degree $\frac{3}{2}$ is well--defined. Also, one can now prove by induction that $\psi_m(x,y,A,B)$ is homogeneous of degree $\frac{1}{2}(m^2-1)$. It follows that $\beta_t(A,B)$ is a homogeneous polynomial of (weighted) degree $\frac{1}{2}(p^2-1)-t$, and hence, it contains only monomials of the form $A^rB^s$ with $2r+3s = \frac{1}{2}(p^2-1)-t$. Hence write $$\beta_t(A,B) = \sum_{2r+3s = \frac{1}{2}(p^2-1)-t} \alpha_{r,s}A^rB^s,$$with $\alpha_{r,s}\in\mathbb{Z}$. We know that $\psi_p$ has leading coefficient $p$ (as a polynomial in $x$), so $\beta_{\frac{p^2-1}{2}} = p$ and hence $\alpha_{0,0} = p$. Also, $\alpha_{r,s} = 0$ if $r<0$ or $s<0$. The following result tells us how, for $t$ close to $\frac{1}{2}(p^2-1)$, the coefficients in $\beta_t$ look like (modulo $p^2$).

\begin{lemma}
For $0 < 2r+3s < q$ we have $$\alpha_{r,s} \in -\frac{(d-1)\left(d-\frac{3}{2}\right)}{d\left(d+\frac{1}{2}\right)}\alpha_{r-1,s} - \frac{\left(d-\frac{3}{2}\right)\left(d-\frac{5}{2}\right)}{d\left(d+\frac{1}{2}\right)}\alpha_{r,s-1} + p^2\mathbb{Z}_{p\mathbb{Z}},$$where $\mathbb{Z}_{p\mathbb{Z}}$ is the localization of $\mathbb{Z}$ by $\mathbb{Z}\setminus p\mathbb{Z}$ (invert everything that is not divisible by $p$).\label{lemma:recursie}
\end{lemma}
\begin{proof}
By \cite[Eq. (3)]{McKee} we know that, for $d = 2r+3s$, $$\begin{aligned} d\left(d+\frac{1}{2}\right)\alpha_{r,s} & = \left(\frac{p^2+3}{2}-d\right)\left(\frac{p^2}{6}-1+d\right)\alpha_{r-1,s} \\ & \quad\quad\quad - \left(\frac{p^2+5}{2}-d\right)\left(\frac{p^2+3}{2}-d\right)\alpha_{r,s-1} \\ & \quad\quad\quad + 3(r+1)p^2\alpha_{r+1,s-1} - \frac{2}{3}(s+1)p^2\alpha_{r-2,s+1}.\end{aligned}$$Hence,
\begin{equation} d\left(d+\frac{1}{2}\right)\alpha_{r,s} = -(d-1)\left(d-\frac{3}{2}\right)\alpha_{r-1,s} - \left(d-\frac{3}{2}\right)\left(d-\frac{5}{2}\right)\alpha_{r,s-1} + p^2w,\label{eq:bah}\end{equation} where $w$ is an expression using $\frac{1}{2},\frac{1}{3}$ and $\alpha_{r',s'}$ with $2r'+3s' < d$. This yields a way to compute $\alpha_{r,s}$ by induction on $d$. To do this, we need to invert $d$ and $2d+1$. Now note that $d = 2r+3s$ is given to be in the set $\{1,2,\ldots,q-1\}$, so $p$ can not divide $d$ or $2d+1< 2q+1 = p$. So using equation (\ref{eq:bah}), and the specific form of $w$, it follows by induction that $\alpha_{r,s}\in \mathbb{Z}_{p\mathbb{Z}}$ for $0 < 2r+3s < q$ (in other words: we don't need to invert $p$ to compute these coefficients). Since $\alpha_{r,s} = 0$ for $2r+3s < 0$, $\alpha_{r,0} = \alpha_{0,s} = 0$ for $r$ and $s$ negative, and $\alpha_{0,0} = p$, we can even say that $\alpha_{r,s}\in \mathbb{Z}_{p\mathbb{Z}}$ for $2r+3s < q$.

Again, using equation (\ref{eq:bah}) and now using the fact that $\alpha_{r',s'}\in \mathbb{Z}_{p\mathbb{Z}}$ for $2r'+3s' < d < q$, we get $$d\left(d+\frac{1}{2}\right)\alpha_{r,s} \in -(d-1)\left(d-\frac{3}{2}\right)\alpha_{r-1,s} - \left(d-\frac{3}{2}\right)\left(d-\frac{5}{2}\right)\alpha_{r,s-1} + p^2\mathbb{Z}_{p\mathbb{Z}}.$$Now use the fact that $d\left(d+\frac{1}{2}\right)$ is not divisible by $p$ to conclude the proof.
\end{proof}

As we noted in the proof, we can use the formula given in the preceding lemma to compute $\alpha_{r,s}$ by induction. This is what we do in the next proposition, in which we solve the above recurrence mod $p$. This is the crux of the proof of the theorem.

\begin{proposition}
If $r$ and $s$ are non--negative integers such that $0 \leq 2r+3s < q$, then $$\alpha_{r,s} \in \left(\frac{-1}{4}\right)^{r+s}\frac{p}{4r+6s+1}\binom{2r+2s}{r+s,r,s} + p^2\mathbb{Z}_{p\mathbb{Z}}.$$
\end{proposition}
\begin{proof}
We will prove this by induction on $d = 2r+3s$, using the formula from Lemma \ref{lemma:recursie}. All the equations below are modulo $p^2\mathbb{Z}_{p\mathbb{Z}}$. (One should be careful not to divide by a multiple of $p$.) In Lemma \ref{lemma:recursie}, we see that $\alpha_{r,s}$ (modulo $p^2\mathbb{Z}_{p\mathbb{Z}}$) is determined by $\alpha_{r-1,s}$ and $\alpha_{r,s-1}$, so the induction goes back to $d' = 2(r-1)+3s = d-2$ and $d'' = 2r+3(s-1) = d-3$. This means that we should have $d\geq 3$, $r\geq 1$ and $s\geq 1$ to use induction. So the first steps of the induction will have to compute $\alpha_{0,0},\alpha_{1,0},\alpha_{0,1}$ (i.e., $\alpha_{r,s}$ with $2r+3s\in\{0,1,2,3\}$), as well as $\alpha_{r,0}$ and $\alpha_{0,s}$ for all non--negative integers $r, s$.
\begin{itemize}
\item We can check the small values to be true, using Lemma \ref{lemma:recursie}. We find $\alpha_{0,0} = p$, $\alpha_{1,0} = -\frac{1}{10}p$ and $\alpha_{0,1} = -\frac{1}{14}p$, which is consistent with our formula.
\item By the recursion formula and $\alpha_{r,-1} = 0$, we know that for $0 < 2r < q$, we have $$\alpha_{r,0} = -\frac{(2r-1)(4r-3)}{2r(4r+1)}\alpha_{r-1,0}.$$Using this repeatedly, we get $$\begin{aligned} \alpha_{r,0} & = \left(\frac{-1}{2}\right)^r \frac{\left[(2r-1)(2r-3)\cdots 1\right]\cdot\left[(4r-3)(4r-7)\cdots 1\right]}{\left[r\cdot (r-1)\cdots 1\right]\cdot\left[(4r+1)(4r-3)\cdots 5\right]}\alpha_{0,0} \\ & = \left(\frac{-1}{2}\right)^r \frac{(2r-1)(2r-3)\cdots 1}{r! (4r+1)}p.\end{aligned}$$ Using the fact that $(2r)! = \left[1\cdot 3\cdots (2r-1)\right]\cdot 2^r\cdot r!$, we find that $$\alpha_{r,0} = \left(\frac{-1}{2}\right)^r \frac{(2r)!}{2^r\cdot (r!)^2 \cdot (4r+1)}p = \left(\frac{-1}{4}\right)^r\frac{p}{4r+1}\binom{2r}{r,r,0},$$which is consistent with our formula.
\item Proving $\alpha_{0,s} = \left(\frac{-1}{4}\right)^s \frac{p}{6s+1}\binom{2s}{s,s,0}$ can be done similarly.
\end{itemize}
So now assume that our equation is true for all $d = 0,1,\ldots,D$ with $D\geq 3$, and suppose $r,s\geq 1$ (because we know it is true for $r = 0$ or $s = 0$). Since $r-1,s-1\geq 0$ and the degrees $2r'+3s'$ in the recursion formula from Lemma \ref{lemma:recursie} are in the interval of the induction hypothesis, we get: $$\begin{aligned} \alpha_{r,s} & = -\frac{(2r+3s-1)(4r+6s-3)}{(2r+3s)(4r+6s+1)}\alpha_{r-1,s} - \frac{(4r+6s-3)(4r+6s-5)}{2(2r+3s)(4r+6s+1)}\alpha_{r,s-1} \\ & = -\frac{(2r+3s-1)(4r+6s-3)}{(2r+3s)(4r+6s+1)}\left(\frac{-1}{4}\right)^{r+s-1}\frac{p}{4r+6s-3}\binom{2r+2s-2}{r+s-1,r-1,s} \\ & \quad\quad- \frac{(4r+6s-3)(4r+6s-5)}{2(2r+3s)(4r+6s+1)}\left(\frac{-1}{4}\right)^{r+s-1}\frac{p}{4r+6s-5}\binom{2r+2s-2}{r+s-1,r,s-1},\end{aligned}$$which a straight--forward computation shows to be equal to $$\left(\displaystyle\frac{-1}{4}\right)^{r+s}\displaystyle\frac{p}{4r+6s+1}\displaystyle\binom{2r+2s}{r+s,r,s}.$$ This proves the proposition.

\end{proof}

Note that we only used $d < q$ when we were dividing by $d+\frac{1}{2}$: we need this not to be a multiple of $p$ as to keep the congruence modulo $p^2\mathbb{Z}_{p\mathbb{Z}}$ true. All the real calculations don't use this assumption $d<q$, so using the proposition we get the following extension:

\begin{proposition}
If $r$ and $s$ are non--negative integers such that $2r+3s = q$, then $$\alpha_{r,s} \in \left(\frac{-1}{4}\right)^{r+s}\binom{2r+2s}{r+s,r,s} + p\mathbb{Z}_{p\mathbb{Z}}.$$
\label{prop:explicitcoef}
\end{proposition}

Note that the factor $p/(4r+6s+1) = p/(2q+1) = 1$ disappeared, and that we only have a congruence modulo $p\mathbb{Z}_{p\mathbb{Z}}$. Also keep in mind that up until now, we were not working in positive characteristic: these formulas say something about the coefficients of $\psi_p(x,A,B)\in\mathbb{Z}[x,A,B]$. From Proposition \ref{prop:explicitcoef} we find $$\begin{aligned} \ell_p(A,B) & = \beta_{\frac{1}{2}p(p-1)}(A,B) = \sum_{2r+3s = q} \alpha_{r,s}A^rB^s \\ & \equiv \sum_{2r+3s = q} \left(\frac{-1}{4}\right)^{r+s}\binom{2r+2s}{r+s,r,s}A^rB^s\pmod{p}.\end{aligned}$$

\subsection{Step 3: equality of coefficients in the sums}

We have proven that $$c_p(A,B) = \sum_{j\in J} \binom{q}{\frac{1}{3}(2q-j),j,\frac{1}{3}(q-2j)}A^jB^{\frac{1}{3}(q-2j)},$$where $J = \left\{j\in\mathbb{Z}\mid j\equiv -q\pmod{3}, 0 \leq j \leq \frac{1}{2}q\right\}$, and $$\ell_p(A,B)\equiv \sum_{2r+3s = q} \left(\frac{-1}{4}\right)^{r+s}\binom{2r+2s}{r+s,r,s}A^rB^s\pmod{p}.$$Note that the indices in this last sum are all couples $(r,s)$ of non--negative integers such that $2r+3s = q$. This condition is equivalent to $r$ and $s = \frac{1}{3}(q-2r)$ being non--negative integers, i.e., $0\leq r\leqslant \frac{1}{2}q$ and $r\equiv -q\pmod{3}$. (For these $r$ and $s$ we have $r+s = \frac{1}{3}(q+r)$.) It follows that $$\ell_p(A,B) \equiv \sum_{j\in J}\left(\frac{-1}{4}\right)^{\frac{1}{3}(q+j)}\binom{\frac{2}{3}(q+j)}{\frac{1}{3}(q+j),j,\frac{1}{3}(q-2j)}A^jB^{\frac{1}{3}(q-2j)}\pmod{p}.$$
Therefore, $c_p(A,B) \equiv \ell_p(A,B)\pmod{p}$ is equivalent to proving $$\binom{q}{\frac{1}{3}(2q-j),j,\frac{1}{3}(q-2j)} \equiv \left(\frac{-1}{4}\right)^{\frac{1}{3}(q+j)} \binom{\frac{2}{3}(q+j)}{\frac{1}{3}(q+j),j,\frac{1}{3}(q-2j)}\pmod{p}$$for all $j\in J$. To prove this, put $j+q = 3k$ with $k\in\mathbb{Z}$ (then $\frac{1}{3}q \leq k \leq \frac{1}{2}q$) and rewrite the congruence as $$\binom{q}{q-k,j,q-2k} \equiv \left(\frac{-1}{4}\right)^k \binom{2k}{k,j,q-2k}\pmod{p}.$$This is equivalent to $$\frac{q!}{(q-k)!} \equiv \left(\frac{-1}{4}\right)^k\frac{(2k)!}{k!}\pmod{p}.$$We rewrite the left hand side as follows: $$\begin{aligned} \frac{q!}{(q-k)!} & = q(q-1)\cdots(q-k+1) = \left(\frac{p-1}{2}\right)\left(\frac{p-3}{2}\right)\cdots \left(\frac{p+1-2k}{2}\right) \\ & \equiv 2^{-k}\cdot (-1)(-3)\cdots(-2k+1) = (-2)^{-k} 1\cdot 3\cdots (2k-1) \\ & = (-2)^{-k}\frac{(2k)!}{2\cdot 4\cdots (2k)} = (-2)^{-k}\frac{(2k)!}{2^k\cdot k!}\pmod{p},\end{aligned}$$which is the desired congruence. This completes the proof of the theorem.

\section{A special curve}

Let $p$ be a prime congruent to 1 modulo 4 and consider the elliptic curve $y^2 = x^3 + x$ over the finite field $\mathbb{F}_p$. Write $p = 4k+1$ with $k\in\mathbb{N}$. Then $c_p(1,0)$ is the coefficient at $x^{p-1} = x^{4k}$ in $(x^3+x)^{2k} = x^{2k}\left(x^2+1\right)^{2k}$, which is clearly $\binom{2k}{k}$. On the other hand, $$\ell_p(1,0) \equiv \sum_{2r+3s = 2k}\left(\frac{-1}{4}\right)^{r+s}\binom{2r+2s}{r+s,r,s}1^r0^s\pmod{p},$$which reduces to $\ell_p(1,0)\equiv \left(\frac{-1}{4}\right)^k\binom{2k}{k,k,0} \equiv (-4)^{-k}\binom{2k}{k}\pmod{p}$. The theorem states that $c_p(1,0) \equiv \ell_p(1,0)\pmod{p}$, which in this case implies that $(-4)^{-k}\equiv 1\pmod{p}$. Using $(-4)^{-1} \equiv k\pmod{p}$ we get

\begin{proposition}
Let $k$ be a positive integer. If $4k+1$ is prime, then it divides $k^k-1$.
\end{proposition}


\noindent \emph{Alternative proof.} Let $p = 4k+1$ be prime. Then 2 is a quadratic residue mod $p$ if and only if $k$ is even, so $(2/p) = 1$ if $k$ is even and $(2/p) = -1$ if $k$ is odd. It follows that $$(-1)^k = \left(\frac{2}{p}\right) \equiv 2^{\frac{p-1}{2}} \equiv 2^{2k} \equiv 4^k\pmod{p},$$so $k^k \equiv (-4k)^k \equiv (1-p)^k \equiv 1\pmod{p}$, as desired. $\qed$

\section*{Acknowledgements}
I would like to thank Antonella Perucca for assisting me when writing my master's thesis, which resulted in this paper.

\end{document}